\documentclass{amsart}
\usepackage[utf8]{inputenc}

\usepackage{mathtools}
\usepackage{centernot}
\usepackage{stmaryrd}
\usepackage{esvect}
\usepackage{upgreek}
\usepackage{comment}
\usepackage{amsfonts}
\usepackage{physics}
\usepackage{tikz-cd}
\usepackage[utf8]{inputenc}
\usepackage[english]{babel}
\usepackage{graphicx}
\usepackage[margin=1in]{geometry}
\usepackage{mathtools}

\usepackage{datetime}
\usepackage{graphicx}
\usepackage[nointegrals]{wasysym}
\usepackage{mathtools}
\usepackage{tikz}
\usepackage{tikz-cd}
\usepackage{enumerate}
\usepackage{marginnote}
\usepackage{amsmath, amsfonts, amssymb, amsthm}
\usepackage{bbm}
\usepackage{hyperref}
\usepackage{booktabs}
\hypersetup{
  colorlinks   = true, 
  urlcolor     = blue, 
  linkcolor    = blue, 
  citecolor   = blue 
}

\usepackage[textwidth=25mm]{todonotes}

\newtheorem{thm}{Theorem}[section]

\newtheorem{lem}[thm]{Lemma}

\newtheorem{prop}[thm]{Proposition}
\newtheorem{cor}[thm]{Corollary}

\theoremstyle{definition}
\newtheorem{example}{Example}

\theoremstyle{definition}

\theoremstyle{remark}

\theoremstyle{definition}
\newtheorem{defn}{Definition}[section]

\usepackage{enumerate}
\usepackage{xcolor}

\usepackage{import}
\usepackage{xifthen}
\usepackage{pdfpages}
\usepackage{transparent}
\usepackage{cite}

\newcommand{\Z}{\mathbb{Z}}

\newtheorem{manualtheoreminner}{Proposition}
\newenvironment{manualprop}[1]{%
  \renewcommand\themanualtheoreminner{#1}%
  \manualtheoreminner
}{\endmanualtheoreminner}

\title{Upper Bounds for Totally Symmetric Sets}
\author{Kevin Kordek, Qiao Li, and Caleb Partin}
\date{\today}

\begin{document}

\maketitle

\begin{abstract}
    Totally symmetric sets are a recently introduced tool for studying homomorphisms between groups. In this paper, we give full classifications of totally symmetric sets in certain families of groups and bound their sizes in others. As a consequence, we derive restrictions on possible homomorphisms between these groups. One sample application of our results is that any homomorphism of a braid group to a direct product of solvable groups must have cyclic image.
\end{abstract}

\section{Introduction}
\label{sec:intro}



\hspace{\parindent} Given any two groups $G$ and $H$, understanding all the possible homomorphisms $G \xrightarrow[]{} H$ is a naturally important but technically difficult task. This motivates using a carefully-chosen subset of the information in the given groups to constrain and eventually classify homomorphisms. As a simple example, cyclic groups and subgroups map to cyclic images. Analogously, the motivation for using totally symmetric sets to study homomorphisms follows this pattern. 

\begin{defn}
A totally symmetric set of a group $G$ is a finite subset $\{x_1, x_2, ..., x_n\}$ of $G$ that satisfies two key properties:
\begin{itemize}
    \item The elements of the subset commute pairwise.
    \item Any permutation of the subset can be achieved through conjugation: for any element of the symmetric group on n letters $\sigma\in S_n$, there exists some $g\in G$ such that $x_{\sigma(i)}=gx_ig^{-1}$ for all $1 \leq i \leq n$. 
\end{itemize}
\end{defn}

\hspace{\parindent} Loosely speaking, a totally symmetric set is a subset of elements in a group characterized by commutativity and conjugation properties, and these properties are preserved under group homomorphisms. Thus, the image of a totally symmetric set under a homomorphism is also a totally symmetric set. An even stronger version of this statement holds: the image of a totally symmetric set under a group homomorphism must be a totally symmetric set of the same size or a singleton. This fundamental lemma, shown in section \ref{section: fundamental lemma and examples}, is a major obstruction to the possible homomorphisms between two groups. 

\hspace{\parindent} Totally symmetric sets are first defined by Kordek and Margalit in the context of the braid group \cite{kordek2019homomorphisms}. The fact that odd Artin generators of the braid group form a totally symmetric set is exploited in characterizing homomorphisms from certain braid group to other certain braid groups. The same fact has been used in the work of the authors along with Chudnovsky, where they derive a lower bound on the size of non-cyclic quotients of the braid group \cite{chudnovsky2020finite}. It has also made an appearance in the work of Caplinger and Kordek \cite{caplinger2020small}, the work of Scherich and Verberne \cite{scherich2020finite}, and the work of Chen and Mukherjea \cite{chen2020braid}. All of these results vitally depend on the properties of totally symmetric sets.

\subsection*{Overview} In this paper, we will develop and present some theory about totally symmetric sets as well as a full classification of them in several families of groups. In Section \ref{section: fundamental lemma and examples}, we give motivating examples of totally symmetric sets and present a proof for the fundamental lemma. In Section \ref{sec:specific-bounds}, we give a full classification of totally symmetric sets in a few special families of groups. In Section \ref{sec: products}, we investigate upper bounds on sizes of totally symmetric sets in various products of groups. In Section \ref{sec: stabilizer}, we give constant bounds on sizes of totally symmetric sets in groups of odd order and solvable groups through the action of the stabilizer. Finally, in Section \ref{section: corollaries}, we present a table of sizes of totally symmetric sets and derive classifications of homomorphisms. 

\subsection*{Acknowledgments} The majority of this work was completed during the summer of 2019 during the Georgia Institute of Technology Mathematics REU under the guidance of Dan Margalit and Kevin Kordek. The authors would like to thank Dan Margalit for his mentorship and guidance throughout this project, Alice Chudnovsky for useful discussions during the Georgia Tech Mathematics REU, and Santana Afton for helpful conversations on free products.
\section{Fundamental Lemma and Examples}
\label{section: fundamental lemma and examples}


\hspace{\parindent}Totally symmetric sets are a powerful aid to the study of homomorphisms due to the following fundamental lemma due to Kordek and Margalit \cite{kordek2019homomorphisms}: 

\begin{lem}[The Fundamental Lemma of Totally Symmetric Sets]
\label{lem:fundamental}
    Let $\varphi:G\xrightarrow{}H$ be a homomorphism between two groups, and $S \subset G$ a totally symmetric set of size $n$. Then $\varphi(S)$ is a totally symmetric set in $H$ with $\lvert \varphi(S) \rvert = n$ or $\lvert \varphi(S) \rvert = 1$.
\end{lem}
\begin{proof}
    We will first show that $\varphi(S)$ is a totally symmetric set of $H$ given that $S$ is a totally symmetric set of $G$. Let $S = \{x_1,x_2,\dots,x_n\} \subset G$. It is clear that the elements of the set $\{\varphi(x_1),\varphi(x_2),\dots,\varphi(x_n)\}$ pairwise commute, so we just need to show they also satisfy the conjugation requirement. Given $\sigma \in S_n$, let $g_\sigma \in G$ be the element inducing the permutation given by $\sigma$ on $S$, i.e. $x_{\sigma(i)} = g_\sigma x_ig_\sigma^{-1}$ for all $i$. Applying $\varphi$ gives $\varphi(x_{\sigma(i)}) = \varphi(g_\sigma)\varphi(x_i)\varphi(g_\sigma)^{-1}$. Thus, $\varphi(g_\sigma)$ achieves the same conjugation for $\varphi(S)$ as $g_\sigma$ does for $S$.
    
    \hspace{\parindent}Now for the second part of the lemma, suppose that $S = \{x_1,x_2,\dots,x_n\}$, and $\varphi(x_1) = \varphi(x_2)$. We need to show that $\varphi(x_1) = \varphi(x_i)$ for all $i$. If $\varphi(x_1) = \varphi(x_2)$, then we know that $\varphi(x_1x_2^{-1}) = e$, and $x_1x_2^{-1} \in \ker\varphi$. Since $S$ is a totally symmetric set we know that there exists an element $h\in G$ such that $hx_1h^{-1} = x_1$ and $hx_2h^{-1} = x_i$. Therefore, $x_1x_i^{-1} = (hx_1h^{-1})(hx_2^{-1}h^{-1}) = hx_1x_2^{-1}h^{-1}$. Since $\ker\varphi$ is a normal subgroup, we know that it is closed under conjugation and $x_1x_i^{-1} \in \ker\varphi$, so $\varphi(x_1) = \varphi(x_i)$. Thus if $S$ does not map injectively to a totally symmetric set of size $n$, then it must map to a set of size 1.
\end{proof}

\hspace{\parindent}An immediate corollary of this lemma is any totally symmetric set of a subgroup $H \leqslant G$ is a totally symmetric set in $G$ by the inclusion map. 

\hspace{\parindent}Totally symmetric sets or their maximum sizes can provide insight about the possible homomorphisms between two groups. Going forward, we denote the maximum size of a totally symmetric set in a group $G$ by $S(G)$. We will discuss this in much greater detail and rigor in section \ref{section: corollaries}. Before diving deep into the theory, we present two primary motivating examples to show what a totally symmetric set looks like in the wild. 

\begin{example}
    The motivating example for totally symmetric sets comes from the braid group, $B_n$. The braid group is generated by a collection of half-twists, $\sigma_1,\sigma_2,\dots,\sigma_{n-1}$, the Artin generators. Take the set of odd Artin generators, 
    \begin{align*}
        X = \{\sigma_{2i-1}\}_{i = 1}^{\lfloor{\frac{n}{2}}\rfloor}.
    \end{align*}
    One can deduce this is a totally symmetric set from the commutation relations of the braid group and the change-of-coordinates principle from mapping class group theory \cite[Section 1.3]{farb2012primer}. This example is especially useful, since the commutator subgroup of the braid group, $B'_n$, is normally generated by $\sigma_1\sigma_3^{-1}$; see \cite{lin2004braids}. From the fundamental lemma of totally symmetric sets above, we know that if $B_n$ maps to a group $G$ with $S(G) < \lfloor \frac{n}{2} \rfloor$, then the homomorphism collapses $X$ to a singleton. Thus $B'_n$ is in the kernel of the homomorphism, and the map factors through the abelianization of $B_n$, which is $\mathbb{Z}$ for all $n$.
    
    \hspace{\parindent}This fact that all homomorphisms from $B_n$ to a group $G$ with $S(G) < \lfloor \frac{n}{2} \rfloor$ are cyclic becomes a powerful tool in studying homomorphisms from the braid group. This is the driving force behind the recent results on braid groups from the authors, Chudnovsky, and Kordek \cite{chudnovsky2020finite}, Kordek and Margalit \cite{kordek2019homomorphisms} and Caplinger and Kordek \cite{caplinger2020small}.
\end{example}

\begin{example}
Another example of large totally symmetric sets comes from the symmetric group $S_n$. Take the standard homomorphism $\varphi: B_n\rightarrow S_n$ where $\sigma_i\mapsto (i \,\,\,\,\, i+1)$. From the fundamental lemma, we know that the image of the totally symmetric set in Example 1 consisting of odd Artin generators will also be a totally symmetric set. Therefore, the set of disjoint transpositions $\{(2i-1 \,\,\,\,\, 2i)\}_{i=1}^{\lfloor\frac{n}{2}\rfloor}$ is a totally symmetric set. We will later see that these two examples are rare instances of large totally symmetric sets, since in the many classifications following this section, most families of groups will have a constant upper bound on the size of their totally symmetric sets.
\end{example}

\section{Classifying Totally Symmetric Sets in Specific Groups}
\label{sec:specific-bounds}

\hspace{\parindent} In this section, we will provide upper bounds for the sizes of totally symmetric sets of free groups, dihedral groups and a subset of the Baumslag--Solitar groups. These results will naturally lead to the bounds given on general products of groups in section \ref{sec: products}.

\hspace{\parindent} Recall that one of the defining conditions of a totally symmetric set is that all of its elements pairwise commute. Thus, a natural place to start when searching for totally symmetric sets is one where all the elements commute, i.e. abelian groups.

\begin{prop}
    \label{prop:abelian}
    Let $A$ be an abelian group and $S \subset A$ a totally symmetric set, then $|S| = 1$. Hence, the only totally symmetric sets of $A$ are all the singleton subsets.
\end{prop}
\begin{proof}
    From the definition of a totally symmetric set, we know that any permutation of the elements of $S$ must be achieved through conjugation in $A$. Since $A$ is abelian, conjugation is always a trivial action, thus the only achievable permutation is the identity. Since all permutations must be possible, we can conclude that $|S| = 1$.
\end{proof}

\hspace{\parindent}It is perhaps disappointing to the reader to start out on such an uninteresting example, but this should motivate us to search elsewhere for interesting totally symmetric sets. At the risk of over-correcting, we can investigate the totally symmetric sets of a highly non-abelian group, the free group.

\begin{thm}
\label{thm:free-groups}
Let $F_2$ be the free group on two generators. For all $G \leqslant F_2$, $S(G) = 1$.
\end{thm}
\begin{proof}
    To prove this theorem, we just need to show that $S(F_2) = 1$, that is, $F_2$ only has trivial totally symmetric sets. 
    
    First, we will check when any two elements in $F_2$ commute. If two elements $a, b \in F_2$ commute, we know that they must generate an abelian subgroup. The Nielsen–Schreier theorem states that any subgroup of a free group is free, and the only abelian free group is $F_1 \cong \mathbb{Z}$. This implies that both of these elements are powers of the generator of $\mathbb{Z}$. If $\mathbb{Z} \cong \langle x \rangle$, then $a = x^n$ and $b = x^m$.
    
    \hspace{\parindent}For $a$ and $b$ to be members of a totally symmetric set, there must exist an element $h \in F_2$ such that $hx^nh^{-1} = x^m$ and $hx^mh^{-1} = x^n$. Take the first of these equalities and raise it to the $m$th power, to obtain the following: 
    \begin{align*}
        (hx^nh^{-1})^m = (x^m)^m = x^{m^2}.
    \end{align*}
    Using our second conjugation equality, we can manipulate the above expression to see that:
    \begin{align*}
        (hx^nh^{-1})^m = hx^{n\cdot m}h^{-1} = (hx^mh^{-1})^n = (x^n)^n = x^{n^2}.
    \end{align*}
    Thus, $x^{n^2} = x^{m^2}$ and $n = \pm m$. This reduces our problem to two distinct cases.
    
    \hspace{\parindent}If $n = m$, then these are the same element and our totally symmetric set consists of a singleton. On the other hand, if $n = -m$, then $hx^nh^{-1} = x^{-n}$ and $hx^{-n}h^{-1} = x^n$. Combining these equalities shows us that $h^2x^nh^{-2} = x^n$, implying $h^2$ commutes with $x^n$ and therefore $h^2$ and $x^n$ can both be expressed as some element $y$ to a power.
    
    \hspace{\parindent}Suppose that $h^2 = y^a$ and $x^n = y^b$. Moreover, we can assume that $y$ is chosen to be minimal in the sense that it cannot be written as another element $z^c$, where $c > 1$. This follows from the residual finiteness of the free groups. Hence, $a$ must be even and $h = y^{a/2}$. Therefore, $h$ commutes with $x^n$, and the equation $hx^nh^{-1} = x^{-n}$ reduces to $x^n = x^{-n}$. This implies that $n = 0$, and thus $F_2$ has trivial totally symmetric sets. 
\end{proof}

\hspace{\parindent}Note since every free group on a finite or countably infinite number of generators is a subgroup of $F_2$, it follows that all free groups have trivial totally symmetric sets. Another consequence is that totally symmetric sets under quotients will not be well-behaved, as every group is a quotient of some free group. 

\hspace{\parindent}From these first two examples, we can make an observation: The two properties that define a totally symmetric set are inherently at odds with each other. On the one hand, we want to find subsets of elements that all pairwise commute, so groups that are ``more abelian" seem to be the natural place to find large sets of this nature. Yet we also need conjugation to be efficacious in our group in order to achieve the full set of permutations, but conjugation is trivial when elements commute. So if we want to find examples of large totally symmetric sets, we need to find groups that intuitively achieve this abelian, non-abelian balance that large totally symmetric sets would require.

\hspace{\parindent}After seeing two important classes of groups that both have trivial totally symmetric sets, we now present our first example of a family of groups exhibiting non-trivial totally symmetric sets.

\begin{thm}
    The maximal size of a totally symmetric subset of the dihedral group $D_{2n}$ is 2 for all $n \geq 3$. Furthermore, every totally symmetric set of size 2 must take the form $\{r^i,r^{-i}\}$ or $\{sr^i,sr^{i+\frac{n}{2}}\}$ for $0 \leq i \leq n-1$, with the latter only occurring if $n$ is divisible by 4.
\end{thm}
\begin{proof}
    We will use the following presentation for the proof: $D_{2n} = \langle r,s\,|\,r^n = 1, s^2 = 1, srs = r^{-1}\rangle$. Hence, any element of $D_{2n}$ can be written as $s^\epsilon r^{i}$ where $\epsilon = 0,1$ and $0 \leq i \leq n-1$. We first check the necessary conditions for any two elements of the group to commute with each other. Any $r^i$ and $r^j$ commute by definition, but for elements $sr^i$ and $sr^j$, we use the following implications to show that if they commute, $i=j$ or $j=i+\frac{n}{2}$:
    \begin{align*}
        sr^isr^j = sr^jsr^i \implies s^2r^{-i}r^j = s^2r^{-j}r^{i} \implies r^{j-i} = r^{i-j}.
    \end{align*}
    Thus, $j - i \equiv i - j \text{ mod n}$, or $2(j - i) \equiv 0 \text{ mod n}$. Either $i = j$, or $j - i = \frac{n}{2} \implies j = i + \frac{n}{2}$, which can only occur when $n$ is even. Therefore, the only nontrivial case of two elements of this form commuting is $sr^i$ and $sr^{i+\frac{n}{2}}$. From the definition of conjugation and the fact that $s^2 = 1$, two conjugate elements must have the same exponents on $s$. Hence, we do not consider any further cases and the two cases discussed above are disjoint from each other. 
    
    \hspace{\parindent}We will now check the conditions necessary for conjugation. The following shows two elements of the form $r^i$ and $r^j$ are conjugate if and only if they are inverses:
    \begin{align*}
        sr^kr^i = r^jsr^k \implies sr^{k+i} = sr^{k-j} \implies k + i = k - j \implies i = -j.
    \end{align*}
    This gives the first totally symmetric set $\{r^i,r^{-i}\}$.  For our other candidates, elements of the form $sr^i$ and $sr^{i + \frac{n}{2}}$ are conjugate if and only if $i = j+\frac{n}{4}$.
    \begin{align*}
        sr^jsr^i = sr^{i+\frac{n}{2}}sr^j \implies s^2r^{i - j} = s^2r^{j - i - \frac{n}{2}} \implies i - j = j - i - \frac{n}{2} \implies 2(i - j) \equiv \frac{n}{2} \text{ mod n }.
    \end{align*}
    This implies that $i - j = \frac{n}{4}$ or $i = j + \frac{n}{4}$, which can only occur if $n$ is divisible by $4$. If two elements of this form are conjugate, there are no other elements in their conjugancy class, as $sr^{(i + \frac{n}{2}) + \frac{n}{2}} = sr^{i}$. We have thus obtained a full classification of the totally symmetric sets of $D_{2n}$, and $S(D_{2n}) = 2$.
    
    \hspace{\parindent}Note: This proof extends to the infinite dihedral group as well $D_\infty = \mathbb{Z} \ltimes \mathbb{Z}_2$, showing that its totally symmetric sets also have a maximal size of 2.
\end{proof}

\hspace{\parindent} The classification in the above proof relies on the convenient presentation of the dihedral group. While this direct proof approach will not work in general, there is another class of groups where we can explicitly compute totally symmetric sets from the group presentation: Baumslag--Solitar groups. For nonzero integers $m,n$, the Baumslag--Solitar group $BS(m,n)$ is defined by the following presentation:
$\langle a,b\,|\,ba^mb^{-1}=a^n \rangle$.
\begin{thm}
    The maximal size of a totally symmetric subset of the Baumslag--Solitar group $BS(1,n)$ is 1 when $n \neq -1$ and 2 when $n = -1$.
\end{thm}

\begin{proof}
    The group $BS(1,n)$ has the one-relation presentation $\langle a,b \,|\, bab^{-1} = a^n \rangle$. This relation is equivalent to $ba = a^nb$, which allows us to write any element of this group in the form $a^ib^j$. We can use this fact to check necessary conditions for different elements of the group to commute and be conjugates. If elements $a^ib^j$ and $a^xb^y$ commute, we have the following implication:
    \begin{align*}
        a^ib^ja^xb^y = a^xb^ya^ib^j \implies a^ia^{xn^j}b^jb^y = a^xa^{in^y}b^jb^y.
    \end{align*}
    This gives the condition that $i + xn^j = x + in^y$. We couple this with necessary conditions for conjugation. If elements $a^ib^j$ and $a^xb^y$ are conjugated to each other by $a^{x_0}b^{y_0}$, we have that
    \begin{align*}
        a^{x_0}b^{y_0}a^ib^j = a^xb^ya^{x_0}b^{y_0} \implies a^{x_0}a^{in^{y_0}}b^{y_0}b^j = a^xa^{x_0n^y}b^yb^{y_0}.
    \end{align*}
    Hence, conjugation give us two conditions: $x_0 + in^{y_0} = x + x_0n^y$ and $y_0 + j = y + y_0 \implies j = y$. Using the latter condition, we can combine this with the conditions for commutativity to obtain the following statement: elements $a^xb^y$ and $a^ib^y$ (as $j = y$) will commute and be conjugate when $i + xn^y = x + in^y$. Rearranging and solving for $i$, we obtain an expression for $i$ in terms of $x,y,n$:
    \begin{equation*}
        i = \frac{x(1-n^y)}{1-n^y}.
    \end{equation*}
    This implies that $i = x$, which makes the two elements the same, except when the denominator is 0. For the denominator to be 0, we have three cases to check. It must either be the case that $y = 0$, $y \neq 0$ and $n = 1$, or $y = 2m$ for some integer $m$ and $n = -1$. 
    
    \hspace{\parindent}If $y = 0$, we only need to check when two elements $a^x$ and $a^y$ are conjugate. By definition of a totally symmetric set there exists an element $h$ swapping $a^x$ and $a^y$ through conjugation: $ha^xh^{-1} = a^y$ and $ha^yh^{-1} = a^x$. We derive conditions on the exponents by raising the former equality to the power of $y$:
    \begin{align*}
        (ha^xh^{-1})^y = a^{y^y} \implies ha^{xy}h^{-1} = a^{y^2} \implies (ha^{y}h^{-1})^x = a^{y^2} \implies (a^{x})^x = a^{y^2}.
    \end{align*}
    So $a^{x^2} = a^{y^2}$ and $x = \pm y$. If $x = y$, then the two elements are the same. Otherwise, if $x = -y$, we instead have $ha^xh^{-1} = a^{-x}$. As previously discussed, we can represent our general element $h$ as the element $b^ia^j$. Plugging this in gives $b^ia^xb^{-i} = a^{-x}$, which becomes $a^{xn^i} = a^{-x}$, so $x(1 + n^{i}) = 0$. Either $x = 0$ or $n^i = -1$. The former case is trivial and the latter can only occur when $n = -1$ and $i$ is an odd integer. Thus, we can examine this case by assuming $n = -1$ and writing the conjugating element as $b^{2m+1}a^j$. This gives $b^{2m+1}a^xb^{-(2m+1)} = a^{-x}$. Using the conjugation relation, we can see that conjugating $a^x$ an odd number of times by $b$ gives $a^{-x}$, so the equation is always true. Hence, when $n = -1$, two elements of the form $a^x$, $a^{-x}$ form a totally symmetric set of size 2.
    
    \hspace{\parindent}The other case to consider is when $y \neq 0$. This breaks down into two subcases, the first of which is $n = 1$. When $n = 1$, the relation for $BS(1,n)$ becomes the commuting relation for $a$ and $b$, implying that $BS(1,1)$ is abelian and thus only has trivial totally symmetric sets. The other case is when $n = -1$ and $y = 2m$ for $m \in \mathbb{Z}$. Thus, we need to check when two elements of the form $a^xb^{2m}$ and $a^yb^{2m}$ commute and are conjugate. Using our general commuting condition from above, we substitute $2m$ for $j$ and $y: i + x(-1)^{2m} = x + i(-1)^{2m}$. This reduces to $i + x = x + i$, which is always true.
    
    \hspace{\parindent}We now must check that for any two elements $a^xb^{2m}$ and $a^yb^{2m}$, there exists an element $h$ that swaps them through conjugation, i.e. $ha^xb^{2m}h^{-1} = a^yb^{2m}$ and $ha^yb^{2m}h^{-1} = a^xb^{2m}$. As above, we raise the former equation to $y$ and simplify:
    \begin{equation*}
        (ha^xb^{2m}h^{-1})^y = (a^yb^{2m})^y \implies ha^{xy}b^{2my}h^{-1} = a^{y^2}b^{2my} \implies (ha^{y}b^{2m}h^{-1})^x = a^{y^2} \implies (a^x)^x = a^{y^2}.
    \end{equation*}
    Therefore, $x^2 = y^2$ and $x = \pm y$. Again, if $x = y$, then the elements are equal, otherwise if $x = -y$, we can write $h$ as $b^ia^j$ to obtain the following: 
    \begin{align*}
        b^ia^ja^xb^{2m}a^{-j}b^{-i} &= a^{-x}b^{2m} \\
        a^{(x+j)(-1)^i}b^{2m+ i}a^{-j}b^{-i} &= a^{-x}b^{2m} \\
        a^{(x+j)(-1)^i + (-j)(-1)^{2m+i}}b^{2m} &= a^{-x}b^{2m} \\
        a^{x(-1)^i}b^{2m} &= a^{-x}b^{2m}.
    \end{align*}
    
    \hspace{\parindent}Thus, the elements $a^xb^{2m}$ and $a^{-x}b^{2m}$ form a totally symmetric set. Any other element in such a totally symmetric set would have to satisfy the above conditions with $a^xb^{2m}$ and $a^{-x}b^{2m}$, which is impossible for an element not equal to either.
\end{proof}

\hspace{\parindent} For the groups above, we are able to directly compute the totally symmetric sets by exploiting the existence of their normal forms. In general, we want to be able to compute the totally symmetric sets of groups without having to rely on assumptions such as the existence of normal forms or explicit descriptions of centralizers and conjugacy classes. For many groups, such a direct computation will not be possible, so we need to introduce more powerful techniques.
\section{Totally Symmetric  Sets in Products of Groups}
\label{sec: products}

\hspace{\parindent} A natural question that arises in studying totally symmetric sets is how they behave under various products, such as direct products, free products, and semi-direct products. To this end, we will prove the following theorem:

\begin{thm}
    \label{ref:product-theorem}
    Let $G$ and $H$ be groups.
    \begin{align*}
        S(G \times H) = \max\{S(G),S(H)\} \\
        S(G * H) = \max\{S(G),S(H)\}
    \end{align*}
\end{thm}

\hspace{\parindent}The next step after understanding how totally symmetric sets behave under direct and free products is to explore semi-direct products or other non-trivial group extensions. While we will not classify all totally symmetric sets in semi-direct products, we will generalize our results on dihedral groups from the previous section.

\begin{prop}
\label{prop:semi-cyclic}
    For $p$ prime and $m$ any integer where $p|m$, $S(\mathbb{Z}_p \ltimes \mathbb{Z}_m) = 2$.
\end{prop}

\subsection{Direct and Free Products}
 



While we may expect that taking the direct product of two groups provides a way of creating larger totally symmetric sets, we begin by showing that direct products do not create larger totally symmetric sets.

\begin{lem}
\label{lem:direct-product}
Let $G$ and $H$ both be groups and $S$ a subset of $G \times H$. If there exists elements $(x_1,y_1),(x_1,y_2),(x_{2},y) \in S$, with $x_1 \neq x_2$, $y_1 \neq y_2$, and $y$ any element in $H$, then $S$ cannot be a totally symmetric subset of $G \times H$.
\end{lem}
\begin{proof}
    Suppose that $S$ is a totally symmetric subset of $G$, and it contains elements of the form above. Since $S$ is a totally symmetric set, we know that there is a permutation $\phi$ that sends $(x_{1},y_{1})$ to $(x_{1},y_{2})$ and $(x_{1},y_{2})$ to $(x_{2},y)$, and there is a conjugating element $(h_G,h_H)$ that achieves this permutation. By the fundamental lemma, the image of $S$ under the projection map $\pi_G$ is a totally symmetric subset of $G$. From the permutation $\phi$, we have $(h_G,h_H)(x_{1},y_{1})(h_G,h_H)^{-1} = (x_{1},y_{2})$ and $(h_G,h_H)(x_{1},y_{2})(h_G,h_H)^{-1} = (x_{2},y)$. Using these equations, we derive a contradiction that $h_G$ conjugates $x_1$ to two different elements. To do this, we first apply $\pi_G$ to both sides of the first equation. This yields:
    \begin{align*}
        \pi_G((h_G,h_H)(x_{1},y_{1})(h_G,h_H)^{-1}) &= \pi_G((h_G,h_H))\pi_G((x_{1},y_{1}))\pi_G((h_G,h_H)^{-1}) \\
        &= h_Gx_{1}h_G^{-1} \\
        &= \pi_G((x_{1},y_{2})) = x_{1}.
    \end{align*}
    Similarly, applying $\pi_G$ to both sides of the second equation, we find that:
    \begin{align*}
        \pi_G((h_G,h_H)(x_{1},y_{2})(h_G,h_H)^{-1}) &= \pi_G((h_G,h_H))\pi_G((x_{1},y_{2}))\pi_G((h_G,h_H)^{-1}) \\
        &= h_Gx_{1}h_G^{-1} \\
        &= \pi_G((x_{2},y)) = x_{2}.
    \end{align*}
    Since $h_G$ conjugates $x_{1}$ to itself and $x_{2}$, we have a contradiction. Therefore, $S$ cannot be a totally symmetric subset of $G\times H$.
\end{proof}

\begin{cor}
\label{cor:directproduct}
Let $G$ and $H$ be groups, and let $S = \{(x_1,y_1),\dots,(x_k,y_k)\} \subseteq G \times H$ be a totally symmetric set. All the $x_i$ are equivalent, or they are all distinct. The same holds for $y_i$. Moreover, $S(G \times H) = \max\{S(G),S(H)\}$.
\end{cor}
\begin{proof}
    By Lemma \ref{lem:direct-product}, if a set $S$ is a totally symmetric subset of $G \times H$, it cannot simultaneously contain elements of the form $(x_1,y_1),(x_1,y_2)$ and $(x_2,y)$. There are thus three possibilities for the elements of $S$:
    \begin{enumerate}
        \item $S$ contains elements of the form $(x_1,y_1)$ and $(x_2,y)$ but none of the form $(x_1,y_2)$. We also can't contain elements of the form $(x_2,y')$ for some $y' \neq b$, since the labels could easily be swapped between the elements. This implies that a totally symmetric set in this case will take the form
        \begin{align*}
            S = \{(x_1,y_1),\dots,(x_n,y_n)\}.
        \end{align*}
        All the $x_i$ are distinct elements of $G$ and each $y_i$ is any element of $H$.
        \item We allow elements of the form $(x_1,y_2)$ and $(x_2,y)$ but none of the form $(x_1,y_1)$. This is equivalent to the previous case by relabeling.
        \item $S$ contain elements of the form $(x_1,y_1)$ and $(x_1,y_2)$, but none of the form $(x_2,y)$. This implies that a set in this case will take the form
        \begin{align*}
            S = \{(x_1,y_1),\dots,(x_1,y_n)\},
        \end{align*}
        where $x_1 \in G$ and all the $y_i \in H$.
    \end{enumerate}
    
    \hspace{\parindent}The above implies that a totally symmetric set $S = \{(x_1,y_1),\dots,(x_k,y_k)\}$ must have all the $x_i$ equivalent or all distinct. The same can be said of the $y_i$ by the symmetry of the direct product.
    
    \hspace{\parindent}Now that we know the totally symmetric sets of $G \times H$ must take on these specific forms, we can first show that $S(G \times H) \leq \max\{S(G),S(H)\}$. Any non-trivial totally symmetric subset of $G \times H$ will have either the $x_i$ or $y_i$ elements all be distinct. Suppose that we have a totally symmetric set $S'$ such that all the $x_i$ elements are all distinct. By Lemma \ref{lem:fundamental}, the set $\pi_G(S')$ is a totally symmetric subset of $G$, in this case with the same cardinality as $S'$ due to all the $x_i$ being distinct. The same can be said if $S'$ had all the $y_i$ distinct and we mapped it to $H$ by $\pi_H$. Together these imply that a totally symmetric subset of $G \times H$ has cardinality bounded above by $\max\{S(G),S(H)\}$.
    
    \hspace{\parindent}To see that $S(G \times H) = \max\{S(G),S(H)\}$, we will construct a totally symmetric subset of size $\max\{S(G),S(H)\}$ in $G \times H$. Without loss of generality, suppose that $S(G) \geq S(H)$ and let $T = \{x_1,\dots,x_n\}$ be a maximal totally symmetric subset of $G$. Consider the set $T \times \{e_H\} \subset G \times H$. This is a totally symmetric subset of $G \times H$. Any two elements $(x_i,e_H)$ and $(x_j,e_H)$ commute. Additionally, any permutation on the elements of $T \times \{e_H\}$ corresponds to a permutation of $T$. If $g_\sigma$ is the element of $G$ that induces a permutation on $T$, then $(g_\sigma,e_H)$ will be the element that induces the same permutation on $T \times \{e_H\}$.
    
\end{proof}

The following corollary will not be used again in the rest of the paper, but it describes totally symmetric sets in direct products if we ask all the elements to be distinct in each coordinate.
\begin{cor}
    Let $G$ and $H$ be groups, and let $S = \{(x_1,y_1),\dots,(x_k,y_k)\} \subseteq G \times H$ be a totally symmetric set. If all the $x_i$ are distinct and all they $y_i$ are distinct, then $|S|=\min\{S(G), S(H)\}$.
\end{cor}

\hspace{\parindent}As shown, totally symmetric sets in direct products of groups behave in a straightforward way. We might now ask whether the same behavior persists in additional group constructions, such as free products, semi-direct products, group extensions. While the latter two are much more difficult to answer, it turns out that we can say the following about totally symmetric sets of the free product of two groups.

\begin{prop}
Let $G$ and $H$ be groups, and let $G*H$ denote their free product. Then $S(G*H) = \max\{S(G),S(H) \}$. Moreover, any totally symmetric subset of $G*H$ is of the form $wSw^{-1}$, where $w \in G*H$, and $S$ is a totally symmetric subset of $G$ or $H$.
\end{prop}

\begin{proof}
    The direction $S(G*H) \geq \max\{S(G),S(H)\}$ is almost immediate. To show this let $S$ be a totally symmetric subset of either $G$ or $H$. Since $G$ and $H$ both inject into their free product, we have that $S$ is also a totally symmetric subset of $G*H$. Thus, we see $S(G*H) \geq \max\{S(G),S(H)\}$. For the other direction, we will show that every totally symmetric set in $G*H$ comes from a totally symmetric subset of either $G$ or $H$ in an injective manner, thus proving $S(G*H) \leq \max\{S(G),S(H)\}$. 
    
    \hspace{\parindent}To begin, let $w_1$ and $w_2$ be commuting elements of $G*H$. Then $w_1$ and $w_2$ are both in the same conjugate of a factor of $G*H$ or both powers of some element $w \in G*H$ \cite[Corollary 4.1.6]{magnus2004combinatorial}. The former of these means that $w_1$ and $w_2$ are both contained in $xGx^{-1}$ or $xHx^{-1}$ for some element $x \in G*H$. This corollary proves an even stronger statement: if we have a set of pairwise commuting elements $\{x_1,\dots,x_n\}$ where only one is in the conjugate of one of the free factors, then all the elements must be in that same conjugate.
    
    Hence, given a totally symmetric set $S$ in $G*H$, there are three possible forms of the elements of the set:
    \begin{enumerate}
        \item \textbf{Conjugates of $G$:} $S = \{wg_1w^{-1},\dots,wg_nw^{-1}\}$ for $g_i \in G$ and $w \in G*H$.
        \item \textbf{Conjugates of $H$:} $S = \{wh_1w^{-1},\dots,wh_nw^{-1}\}$ for $h_i \in H$ and $w \in G*H$.
        \item \textbf{Powers of the same element:} $S = \{v^{i_1},\dots,v^{i_n}\}$ for $v \in G*H$ and $i_1 \leq i_2 \leq \dots \leq i_n$. 
    \end{enumerate}
    
    \hspace{\parindent}We will first examine the first two of these three cases. Without loss of generality, suppose that we have a totally symmetric set $S$ of the first type, with all its elements in a conjugate of $G$. Since this is a totally symmetric set, any permutation of its elements can be achieved through conjugation. Let $\sigma$ be a permutation of the elements of $S$ and let $w_\sigma \in G*H$ be an element of $G*H$ that achieves this permutation by conjugation. If $\sigma(i) = j$, then $w_\sigma wg_1w^{-1}w_\sigma^{-1} = wg_jw^{-1}$, which implies $w_\sigma = wg_\sigma w^{-1}$ and $g_\sigma g_i g_\sigma^{-1} = g_j$ for $g_\sigma \in G$ . This combined with commutativity of $g_i$ shows $S' = \{g_1,\dots,g_n\}$ is a totally symmetric subset of $G$ with the same size as $S$. Therefore, any totally symmetric subset of the first two forms above can have size at most $\max\{S(G),S(H)\}$.
    
    \hspace{\parindent}In the other case, all elements of $S$ are powers of the same element $v\in G*H$. We can assume $v$ is not an element of a conjugate of one of the factors, since otherwise, all powers of it would be as well, which would then reduce to case $(1)$ and $(2)$. Moreover, all elements with torsion in a free product are conjugates of finite-order elements in one of the factors \cite[Corollary 4.1.4]{magnus2004combinatorial}, implying $v$ must have infinite order. Since our set $S = \{v^{i_1},\dots,v^{i_n}\}$ is totally symmetric, there exists an element $w \in G*H$ such that $wv^{i_1}w^{-1} = v^{i_2}$, $wv^{i_2}w^{-1} = v^{i_1}$, and $w$ fixes all other element of $S$. This reduces to a similar situation as in the proof of Theorem \ref{thm:free-groups} and it has a similar solution. We can show that $i_2 = \pm i_1$ by taking $wv^{i_1}w^{-1} = v^{i_2}$, and exponentiating both sides to the $i_2$th power:
    \begin{align*}
        (v^{i_2})^{i_2} = (wv^{i_1}w^{-1})^{i_2} = wv^{i_1i_2}w^{-1} = (wv^{i_2}w^{-1})^{i_1} = (v^{i_1})^{i_1}.
    \end{align*}
    This implies $v^{i_2^2} = v^{i_1^2}$, so $i_2^2 = i_1^2$ and $i_2 = \pm i_1$. In either case we have now reduced our totally symmetric set $S$ to either a singleton or $\{v^{i_1},v^{-i_1}\}$, as permutations such as the one above must be possible for all pairs of elements in $S$. 
    
    \hspace{\parindent}We now check if $\{v^{i_1},v^{-i_1}\}$ is a totally symmetric set. Again, we must have an element $w \in G*H$ such that $wv^{i_1}w^{-1} = v^{-i_1}$, and $wv^{-i_1}w^{-1} = v^{i_1}$. As we saw in the proof of Theorem \ref{thm:free-groups}, this implies that $w^2$ commutes with $v^{i_1}$. From above, either $w^2$ and $v^{i_1}$ are in the same conjugate of one of the factors of $G*H$, or they are powers of the same element. If $v^{i_1}$ was a conjugate of an element in either $G$ or $H$, $v$ is also a conjugate of an element in $G$ or $H$, contradicting our initial assumption. Therefore, $v^{i_1} = z^a$ and $w^2 = z^b$ for some element $z \in G*H$. We can assume $z$ is primitive, that it can't be written as a power of another element in $G*H$. 
    
    \hspace{\parindent}We now show that $b$ is even and $w = z^{\frac{b}{2}}$. Since $w$ is a freely reduced word in $G$ and in $H$, it either has even or odd length. If $w$ has even length, then its starting and ending elements cannot both be from the same group. Thus, $w^2$ is exactly the concatenation of $w$ with itself, with no cancelling or reducing involved. Since $z$ is reduced and primitive, $w$ is a concatenation of $\frac{b}{2}$ copies of $z$. Alternatively, suppose that $w$ has odd length, implying it ends and begins with an element from the same factor. If cancellation occurs when concatenating, $w$ and $z$ are both conjugates of the same element $x\in G*H$: $z=xz'x^{-1}, w=xw'x^{-1}$. This reduces to proving that $w'=z'^{\frac{b}{2}}$, where $w', z'$ cannot be written as conjugates of some shorter length element. Thus, assume $w$ and $z$ are not conjugates of a smaller length element. When taking powers of both elements, the only reduction will be happening at the letters where they are concatenated, so we can conclude that $w = z^\frac{b}{2}$.
    
    \hspace{\parindent}Since $w$ and $v^{i_1}$ are both powers of $z$, $w$ commutes with $v^{i_1}$. This implies $v^{i_1} = v^{-i_1}$ and thus, $S$ is a singleton. Hence, there are no non-trivial totally symmetric sets of form $(3)$, and all totally symmetric sets of a free product come from conjugates of totally symmetric sets in the factors.
\end{proof}
\hspace{\parindent} We can now see that totally symmetric sets behave in a reasonably nice way with regards to these basic product operations and have a full proof of Theorem \ref{ref:product-theorem}. However, before we rejoice over this predictable behavior of totally symmetric sets under products, it's worth pointing out a disappointment: at this point, we still don't have a way to construct larger totally symmetric sets from small ones on the group level: The maximum size of totally symmetric sets under direct and free products are limited by the totally symmetric sets within its factors. It has been a theme throughout this paper that large totally symmetric sets are rare. 


\subsection{Semi-direct Product of Cyclic Groups}

\begin{lem}
\label{lem:inverses}
Let $G$ be a group and $S \subset G$ a totally symmetric subset of $G$ such that $g, g^{-1} \in S$ for some $g \in G$. Then $S = \{g, g^{-1}\}$.
\end{lem}
\begin{proof}
    Label the elements of the totally symmetric set $S$ as $S = \{g, g^{-1}, x_1, \dots, x_n\}$ for some non-negative integer $n$. Since $S$ is a totally symmetric set, there exists an element $h \in G$ such that $hgh^{-1} = g$ and $hg^{-1}h^{-1} = x_i$ for one of the $x_i \in S$. Inverting both sides of the former of these equalities shows $hg^{-1}h^{-1} = g^{-1}$, thus we must have that $x_i = g^{-1}$. This process can be continued for all elements of $S$, hence, $S = \{g, g^{-1}\}$.
\end{proof}

\begin{manualprop}{\ref{prop:semi-cyclic}}
    For $p$ prime and $m$ any integer where $p|m$, $S(\mathbb{Z}_p \ltimes \mathbb{Z}_m) = 2$.
\end{manualprop}

\begin{proof}
    We use the presentation $\mathbb{Z}_p \ltimes \mathbb{Z}_m=\{s,r|r^p=e,s^m=e, srs^{-1}=r^k\}$ for the proof. A typical element in the group is of the form $r^as^b$, where $a\in\mathbb{Z}/p\mathbb{Z}$ and $b\in\mathbb{Z}/m\mathbb{Z}$. We first investigate the conditions for two elements, $r^as^b$ and $r^xs^y$, to commute.:
    \begin{equation*}
    \begin{split}
        r^as^br^xs^y &= r^xs^yr^as^b \\
        r^{a-x}s^br^s &= s^yr^as^{b-y}\\
        r^{a-x}s^br^s &= s^yr^as^{-y}s^b\\
        r^{a-x}s^br^s &= r^{ak^y}s^b\\
        r^{a-x-ak^y}(s^br^xs^{-b}) &= e\\
        r^{a-x-ak^y}r^{xk^b} &= e\\
        r^{a-x-ak^y+xk^b} &= e.\\
    \end{split}
    \end{equation*}
     
    \hspace{\parindent}By the group presentation, we obtain the following equation:
    \begin{equation}
        a-x-ak^y+xk^b = 0\mod p.
    \end{equation}

     \hspace{\parindent}We put this aside and discuss the conditions for the same two elements, $r^as^b$ and $r^xs^y$, to also be conjugates of each other. Before we go through the calculations, we take a detour to discuss $s^{-1}rs$. By $srs^{-1}=r^k$, it follows that $s^{-1}rs = r^l$, where $l$ is the multiplicative inverse of $k$ modulo $p$, which exists because $p$ is prime. Suppose there exists some element $r^es^f$ that conjugates $r^as^b$ to $r^xs^y$. We can use the group presentation to show that the equation $b+y=0\mod m$ must hold:
    \begin{equation*}
    \begin{split}
        r^es^fr^as^bs^{-f}r^{-e} &= r^xs^y\\
        r^es^bs^fs^{-b}r^as^bs^{-f}r^{-e} &= r^xs^y\\
        r^es^bs^fr^{al^b}s^{-f}r^{-e} &= r^xs^y\\
        r^es^br^{al^bk^f-e} &= r^xs^y\\
        r^es^br^{al^bk^f-e}s^{-b}s^{-y}s^b &= r^x\\
        r^{e+(al^bk^f-e)k^b}s^{b+y}&=r^x.
    \end{split}
    \end{equation*}
    
    \hspace{\parindent}Deciphering the exponents, we now have the following relations: $b+y=0\mod m$, and $e+(al^bk^f-e)k^b=x\mod p$. Recall $b, y \in \mathbb{Z}/m\mathbb{Z}$ and thus, $ 0 \leq b, y < m$. The relation $b+y=m\mod m$ then implies that either $b = y = 0$, or $y = m - b$. 
    
    \hspace{\parindent} The case that $b=y=0$ amounts to considering elements of the form $r^a$ and $r^x$ as candidate elements of the totally symmetric set. The commuting relation $(1)$ shows elements of this form always commute, and the second of the two conjugation relations reduces to $ak^f \equiv x\mod p$ given $b=0$ and $k^b \equiv 1\mod p$. Therefore, choosing a conjugating element $r^es^f$ determines $x$, and the relation shows the choice of $e$ will not affect the conjugation. It is important to note that because $f$ is an integer between $0$ and $m-1$, multiple choices of $f$ could lead to the same $x$. Let $0 \leq f \neq f' \leq m-1$ such that $k^{f} \equiv k^{f'} \mod p$. This implies that $k^{f - f'} \equiv 1\mod p$ and thus, $f$ and $f'$ must differ by a multiple of $p-1$. 
    
    \hspace{\parindent} Choose $f$ as above to determine an $x$, and let $f' = f + n(p-1)$. We conclude by the following calculation that $r^es^{f'}$ acts on $r^x$ by conjugation in the same way that $r^es^{f}$ does, i.e. different choices of $f$ that determine the same $x$ will act the same by conjugation: 
    \begin{align*}
        r^{e}s^{f'}r^{x}s^{-f'}r^{-e} = r^{e}r^{xk^{f'}}r^{-e} = r^{xk^{f + n(p-1)}} = r^{xk^{f}}.
    \end{align*}
   
    \hspace{\parindent} In order for $r^a$ and $r^x$ to form a totally symmetric set, we need an element that conjugates $r^a$ to $r^x$ to also conjugate $r^x$ back to $r^a$. Hence, we need $xk^f \equiv a\mod p$. Combining this with $ak^f \equiv x\mod p$, we have that $k^{2f} \equiv 1\mod p$ and thus, $k^f \equiv \pm 1\mod p$. This tells us that $x = \pm a$, and one cannot make a larger totally symmetric set containing the elements $r^a$ and $r^x$.
    
    \hspace{\parindent}For the other case, assume that $y = m - b$, and use $(1)$ to derive an equation to determine the following equations expressing $x$ in terms of $a$ and $b$:
    \begin{equation} 
    \begin{split}
        a-x-ak^{m - b}+xk^b = 0 \mod p \\
        x(k^b - 1) = a(k^{m-b} - 1) \mod p. \\
    \end{split}
    \end{equation}

    \hspace{\parindent}If $k^b-1\neq 0\mod p$, then given $a,b$, there is a unique solution for $x$. Otherwise, we consider the case where $k^b - 1 \equiv 0\mod p$.
    
    \hspace{\parindent} From above we know that $x = \pm a$ and thus we have the following three elements to consider for a totally symmetric set, $r^as^b, r^as^{-b}$, and $r^{-a}s^{-b}$. We know that we can construct two totally symmetric sets from these elements, namely $\{r^{a}s^b, r^as^{-b}\}$ and $\{r^{a}s^b, r^{-a}s^{-b}\}$, but we also need to check if all three elements can form a totally symmetric set. Under our assumption that $k^b \equiv 1\mod p$, we can show that $r^as^b$ and $r^{-a}s^{-b}$ are inverses:
    \begin{align*}
        r^as^br^{-a}s^{-b} = r^a(s^brs^{-b})^{-a} = r^{a}(r^{k^b})^{-a} = r^{a}r^{-a} = e.
    \end{align*}
    
    By Lemma \ref{lem:inverses}, this inhibits all three elements from being in a totally symmetric set together, so the maximal size of a totally symmetric set in $\mathbb{Z}_p \ltimes \mathbb{Z}_m$ is 2.
    
\end{proof}

\section{The Stabilizer}
\label{sec: stabilizer}

\hspace{\parindent}In this section, we introduce the stabilizer of the totally symmetric set and use it to bound the cardinality of totally symmetric sets in groups. In particular, we will show that the totally symmetric sets in odd-ordered finite groups and solvable groups are small.\\

\hspace{\parindent}Let $G$ be a group, $S$ a totally symmetric subset of size $n$ in $G$, and $\text{Stab}_G(S)$ the stabilizer of $S$ in $G$ under conjugation. We claim that $\text{Stab}_G(S)$ surjects onto $S_n$, the symmetric group on $n$ letters. The stabilizer acts on $S$ by conjugation, so we can explicitly construct a homomorphism $\varphi: \text{Stab}_G(S)\rightarrow Sym(S)\cong S_n$ described by sending $\gamma \in\text{Stab}_G(S)$ to the automorphism of $S$ given by $s\rightarrow \gamma s\gamma^{-1}$. By definition, every permutation of $S$ is realized as a conjugation by some $\gamma \in \text{Stab}_G(S)$, so this homomorphism is surjective. Equivalently, we can write this as a short exact sequence, where $k$ denotes the kernel of $\varphi$:
$$
1\xrightarrow{}k\xrightarrow{}\text{Stab}_G(S)\xrightarrow{}S_n\xrightarrow{}1
$$

\hspace{\parindent}In the case that $G$ is of finite order, $|G|\geq|\text{Stab}_G(S)|=|k||S_n|=n!|k|$. The kernel consists of elements that commute with all elements of the totally symmetric set, or equivalently the intersection of centralizers of each element. In particular, the subgroup generated by the totally symmetric set will be a subgroup of the kernel $k$ by commutativity. We can relate the cardinality of the group $G$ to the size of its totally symmetric sets through bounding the size of the kernel. Using a similar technique, Chudnovsky and the authors showed that for braid groups and their commutator subgroups, the required cardinality of $G$ grows super-exponentially as $n$ increases \cite{chudnovsky2020finite}.

\hspace{\parindent}We use the existence of the surjection from the stabilizer to the symmetric group to bound the sizes of totally symmetric sets in the following groups.

\begin{prop}
\label{prop:odd-order}
    Let $G$ be a finite group with odd order. Then $S(G) = 1$.
\end{prop}

\begin{proof}
    Let $|G| = 2n + 1$ for some $n \in \mathbb{N}$, and let $S\subset G$ be a totally symmetric set in $G$. From the short exact sequence above, the order of $Sym(S)$ must divide the order of $\text{Stab}_G(S)$. Moreover, since the order of $\text{Stab}_G(S)$ divides the order of the whole group $G$, we have that $|S|! \,\,|\,\, |G|$ and hence $|S|! \,\,|\,\, 2n + 1$. Since for $|S| \geq 2$, $|S|!$ is even, we must have that $|S| = 1$. 
\end{proof}

\begin{thm}
\label{thm:solvable}
    Let $G$ be a solvable group. Then $S(G) \leq 4$.
\end{thm}
\begin{proof}
    Let $S \subset G$ be a totally symmetric subset of $G$. Since $\text{Stab}_G(S)$ is a subgroup of $G$, and $G$ is solvable, $\text{Stab}_G(S)$ is also solvable. Moreover, from the short exact sequence above, we know that $Sym(S)$ is a quotient of $\text{Stab}_G(S)$, so it is solvable. For $n \geq 5$, $S_n$ is not solvable. Thus, $|S| \leq 4$.
\end{proof}

We do not know if this bound is sharp. In other words, it's an open question whether there are totally symmetric sets in solvable groups of size 3 or 4.

\section{Corollaries for Homomorphisms}
\label{section: corollaries}

\hspace{\parindent} In this last section, we summarize the results and display various corollaries on possible homomorphisms between the groups discussed in this paper. The corollaries on the homomorphisms we present stem from the following two corollaries of the fundamental lemma of totally symmetric sets:

\begin{cor}
Let $G$ and $H$ be groups with $S(G) > S(H)$. No homomorphism $f: G \xrightarrow[]{} H$ can be injective.
\end{cor}
\begin{proof}
    Let $S \subset G$ be a totally symmetric set of size $S(G)$ and $f:G \xrightarrow[]{} H$ be a homomorphism. From the fundamental lemma, $S$ maps to a totally symmetric set in $H$ of size $S(G)$ or a singleton. Since $S(G) > S(H)\geq 1$, there are no totally symmetric sets in $H$ of size $S(G)$, and $S$ maps to a singleton. Since $S$ is not a singleton,  $f$ is not injective.
\end{proof}

\hspace{\parindent} This corollary tells us that $S(G)$ serves as an obstruction to $G$ being a subgroup of another group. A further corollary that has been used frequently in the recent results on braid groups using totally symmetric sets is the following:

\begin{cor}
\label{cor:cyclic}
Let $B_n$ be the braid group on $n$ strands for $n \geq 5$ and let $G$ be a group such that $S(G) < \lfloor \frac{n}{2} \rfloor$, then any homomorphism $f:B_n \xrightarrow[]{} G$ is cyclic, i.e. it factors through $\mathbb{Z}$.
\end{cor}

\hspace{\parindent} A refinement of this statement using totally symmetric sets by Chudnovsky and the authors \cite{chudnovsky2020finite} and improved by Caplinger and Kordek \cite{caplinger2020small} says that this is true if $|G| < \lfloor \frac{n}{2} \rfloor! \cdot 3^{\lfloor \frac{n}{2} \rfloor - 1}$. This bound was recently improved upon by Scherich and Verberne \cite{scherich2020finite}.

\hspace{\parindent} The table below summarizes the bounds on totally symmetric sets derived in this paper. Applying the above two corollaries gives us a list of groups which can have no injective homomorphisms between them. For example, $\Z_p\ltimes \Z_{mp}$ cannot be a subgroup of the Baumslag--Solitar group $BS(1,n)$. Likewise, we obtain many example of groups that the braid groups only map cyclically to. For example, by Corollary \ref{cor:cyclic} any homomorphism from $B_n$ where $n\geq 10$ to solvable groups must have cyclic image. Using the last line, we can take arbitrary products of the groups in this table, and $B_n$ with $n\geq 10$ will map cyclically to it. For example, any homomorphism from $B_n$ into $(BS(1,420)*(\Z_{2017}\ltimes \Z_{8068}))\times D_{914}$ has cyclic image.

\vspace{1.5em}
\centering
\begin{tabular}{@{} *5l @{}}    \toprule
$S(G)$ & Group &&&  \\\midrule
1    & Abelian    \\ 
     & Free Group ($F_n$)  \\
     & Odd Order Group \\
     & $BS(1,n)$, $n \neq 1$ \\
2    & Dihedral ($D_{2n}$)   \\ 
     & $\mathbb{Z}_p \ltimes \mathbb{Z}_{np}$  \\
     & $BS(1,-1)$ \\ 
$\leq 4$ & Solvable Groups \\ 
Max$\{S(G),S(H)\}$    & Direct Product, $G \times H$    \\ 
                      & Free Product, $G*H$  \\\bottomrule
 \hline
\end{tabular}

\bibliographystyle{alpha}
\bibliography{references}

\end{document}